\documentclass[11pt]{amsart}
\usepackage[mathscr]{euscript}
\usepackage{amsthm}
\usepackage{amssymb}

\swapnumbers
\theoremstyle{plain}
\newtheorem{theorem}{Theorem}[section]
\newtheorem{proposition}[theorem]{Proposition}
\newtheorem{lemma}[theorem]{Lemma}
\newtheorem{corollary}[theorem]{Corollary}

\theoremstyle{definition}
\newtheorem{definition}[theorem]{Definition}
\newtheorem{example}[theorem]{Example}

\newtheorem{remark}[theorem]{Remark}

\newcommand{\Z}{\mathbb{Z}}
\newcommand{\N}{\mathbb{N}}
\newcommand{\Q}{\mathbb{Q}}
\newcommand{\R}{\mathbb{R}}
\newcommand{\C}{\mathbb{C}}
\newcommand{\F}{\mathbb{F}}
\newcommand{\sub}{\subseteq}
\newcommand{\gen}[1]{\left\langle #1\right\rangle}

\def\I{\mathbf{i}}
\def\P{\textup{\textbf{P}}}
\newcommand{\set}[2]{\{#1| #2\}}
\newcommand{\eps}{\varepsilon}
\renewcommand{\phi}{\varphi}

\begin{document}
\title{The lowest-degree polynomials with non-negative coefficients}

\author[T.~Kepka]{Tom\'{a}\v{s}~Kepka}
\address{Charles University, Faculty of Mathematics and Physics, Department of Algebra \\
Sokolovsk\'{a} 83, 186 75 Prague 8, Czech Republic}
\email{kepka@karlin.mff.cuni.cz}

\author[M.~Korbel\'a\v r]{\textsc{Miroslav Korbel\'a\v r}}
\address{Department of Mathematics and Statistics, Faculty of Science, Masaryk University, Kotl\' a\v rsk\'{a} 2, 611 37 Brno, Czech Republic}
\email{miroslav.korbelar@gmail.com}

\thanks{The first author was supported by the Grant Agency of the Czech Republic, grant \#201/09/0296. The second author was supported by the project LC 505 of Eduard \v Cech's Center for Algebra and Geometry.}


\begin{abstract}
A polynomial $p\in\R[x]$ is a divisor of some polynomial $0\neq f\in\R[x]$ with non-negative coefficients if and only if $p$ does not have a positive real root. The lowest possible degree of such $f$ for a given $p$ is known for quadratic polynomials. We provide it for cubic polynomials and improve known bounds of this value for a general polynomial.
\end{abstract}

\maketitle

\section{Introduction}

Polynomials with non-negative coefficients appear naturally in the ana\-ly\-sis of stochastic context-free grammars (with applications to natural language processing \cite{geman, manning}), control theory (testing of stable polynomials), optimization and semialgebraic geometry (application of P\'olya's theorem) and algorithmic game theory (bounding the price of anarchy). They are also of interest in approximation theory \cite{nussbaum,yu},  rewriting systems \cite{lucas,zantema}, number theory \cite{kuba,steinberger} or graph theory \cite{woestijne} and have application in biochemistry \cite{briggs,xia} and electronics \cite{riblet}. We will mention here yet another natural motivation for studying of these polynomials - the ordered domains and their connection with semirings.


A commutative and unitary domain $R$ with a non-identical (partial) order $\preceq$ on $R$ is called \emph{ordered} if
(1) $a\preceq b \Rightarrow a+c\preceq b+c$ and (2) $(a\preceq b \;\&\; 0\preceq c) \Rightarrow ac\preceq bc$
for all $a,b,c\in R$. In real algebraic geometry the ordered domains with additional condition $0\preceq a^{2}$ for all $a\in R$ are studied. This property is a weaker form of the linear ordering, but still there are a lot of rings (e.g. algebraically closed fields), where the last mentioned condition can not be fulfilled, although the definition using only (1) and (2) provides many non-linear orders generally.

In an ordered ring $R$ (necessary of characteristic $0$) is the order $\preceq$ uniquely determined by the set $\P(\preceq)=\{a\in R \;|\; 0\prec a\}$ of positive elements. This set is a \emph{subsemiring of $R$} (i.e. non-empty, closed under addition and multiplication) and does not contain zero. For $a\in R$ there exists an order $\preceq$ on $R$ such that $0\prec a$ if and only if the semiring $P=\set{f(a)}{0\neq f\in x\cdot \N_{0}[x]}$ does not contain $0$. Since $\mathrm{char}(R)=0$, we can consider the minimal polynomial $0\neq\min_{\Q}(a)\in\Z[x]$ of $a$, if it exists (otherwise put $\min_{\Q}(a)=0$). Then we get the following characterization:

\begin{theorem}\label{B1}
 Let $R$ be a commutative domain and $a\in R$. Then there exists at least one order  $\preceq$ on $R$ such that $0\prec a$ if and only if the polynomial $\min_{\Q}(a)$ has a real positive root.
\end{theorem}

This theorem is an immediate consequence of the following assertion:

\begin{theorem}\cite{poincare,meissner,polya}\label{B2}
Let $p\in\R[x]$ be a polynomial with no positive real roots. Then there is $0\neq h\in\Q[x]$ such that the polynomial $f=hp$ has all its coefficients non-negative.
\end{theorem}

Probably the first known proof of \ref{B2} appeared implicitly in papers of Poincar\'e \cite{poincare} and Meissner \cite{meissner}. Moreover, by P\'olya's theorem \cite{polya} one can choose $h=\pm(x+1)^{k}$ with $k\in\N$ sufficiently large in \ref{B2}.

Apart of this, for that $a\in R$ which does not allow the appropriate ordering, the structure of the $\Q^{+}[x]$-semimodule $M=\set{f\in\Q^{+}[x]}{f(a)=0}\sub\Q^{+}[x]$ can be studied (here $\F^{+}$ is the semiring of usual non-negative numbers of a field $\F\sub\R$ and $\F^{+}[x]$ is the semiring of polynomials over $x$ with coefficients from $\F^{+}$). In fact, this is a particular version of a general problem:

\emph{Having a field $\F\sub\R$ and an ideal $I\sub\F[x]$, determine the $\F^{+}[x]$-semimodule $I^{+}:=I\cap \F^{+}[x]$.}

In this paper we recall and combine known results in connection with this problem, especially concerning the estimations of the minimal degree of all non-zero elements in $I^{+}$, and make an improvement of them. It is worth to mention that despite the wide range of applications, the polynomials with non-negative coefficients are not studied systematically and  many papers concerning them were done by authors working in different fields of mathematics and using various terminologies. This probably is the reason why the results do not follow up each other, miss citations and several of them were "rediscovered" repeatedly (e.g. Theorem \ref{B2}).


\section{Polynomials with non-negative coefficients}\label{sec2}

Through this paper, let for  a semiring $S$ and  an $S$-semimodule $M$ be $\gen{X}_{S}$ the subsemimodule of $M$ generated by a set $X\sub M$.

Let $\F$ be a subfield of $\R$, $p\in\F[x]$ and $I=p\F[x]$.  First recall, that for a given $n\in\N_{0}$ the set $M_{n}=\set{f\in I^{+}}{\deg(f)\leq n+ \deg(p)}$ is an intersection of the vector space $\gen{p,xp,\dots,x^{n}p}_{\F}$ over $\F$ with the convex $\F^{+}$-cone $\set{f\in \F^{+}[x]}{\deg(f)\leq n+ \deg(p)}$. Using basic knowledge of convex geometry one without difficulty gets  that $M_{n}$ is thus generated by a finite set of "extremal rays" or more precisely, there is the smallest (up to $\F^{+}$-multiples of elements) finite generating set $X_{n}\sub\F^{+}[x]$ of the $\F^{+}$-semimodule $M_{n}$. These generators can be computed via linear equations using only rational numbers and coefficients of $p$ and therefore they do not depend on the choice of the extension field. Similarly, there is the least generating set $X$ (up to $\F^{+}$-multiples of elements) of the $\F^{+}[x]$-semimodule $I^{+}$ ($X$ may be obtained from $\cup_{n\in\N_{0}}X_{n}$ by letting out those polynomials that are $x^{k}$-multiples (for $k\in\N$) of other polynomials in $\cup_{n\in\N_{0}}X_{n}$).  It is also not difficult to show that $I^{+}$ is finitely generated as an $\F^{+}[x]$-semimodule if and only $I$ is a monomial ideal.

For a chosen polynomial it is therefore possible (more or less efficiently)  to compute generators of $I^{+}$ up to any given degree. But to determine  explicitly $I^{+}$ in general seems to be a difficult task, since even the estimations of the least possible degree of a non-zero polynomial in $I^{+}$ are known only very roughly.

\begin{definition}
For $0\neq p\in\R[x]$ denote $\deg_{+}(p)$ the least possible degree of a non-zero polynomial in $\R^{+}[x]$ divisible by $p$ (if such exists), otherwise put $\deg_{+}(p)=\infty$. Further put $D(p)=\deg_{+}(p)-\deg(p)$ and $\mathcal{D}^{+}(\R)=\{0\}\cup\set{f\in\R[x]}{(\forall a\in(0,\infty))\ f(a)>0}$.
\end{definition}

Polynomials of minimal degree in $(p\R[x])^{+}$ for $p$ being quadratic  were  determined by Dancs \cite{dancs}. His proof used Minkowski-Farkas theorem for systems of linear inequalities. This result was later proved by Motzkin and Strauss \cite{strauss} by a simpler argument using cones of complex numbers. Finally, Harnos \cite{harnos} described completely the generating set of the $\R^{+}[x]$-semimodule $(p\R[x])^{+}$ for $p\in\R[x]\setminus\R^{+}[x]$ being quadratic.

\begin{theorem}\cite{dancs,harnos}\label{B4}
 Let $p=x^2+ax+b=(x-\omega)(x-\overline{\omega})\in\R[x]$, where $\omega=r(\cos \varphi+\I\sin\varphi)$, $r>0$ and $\varphi\in(0,\pi)$. Then
$$\deg_{+}(p)=\Big\lceil\frac{\pi}{\varphi}\Big\rceil=\Big\lceil\frac{\pi}{\arcsin(\sqrt{1-\frac{a^{2}}{4b}})}\Big\rceil.$$


Moreover, there is $f\in(p\R[x])^{+}$ with all coefficients positive such that $\deg(f)=\deg_{+}(p)+1$.
\end{theorem}

We derive the value of $\deg_{+}$ for cubic polynomials. Let us recall that every  polynomial $p$ of degree $d\geq 1$ has a well known correspondence to a linear homogeneous recurrence relation with constant coefficients. All sequences that are solution of this relation form a vector space of dimension $d$ with a particular basis $(a^{(1)}_{i})_{i=0}^{\infty},\dots,(a^{(d)}_{i})_{i=0}^{\infty}$ made up with help of roots of $p$ and derivations.
 Now put $(\mathbf{v}_{k}(p)=)\mathbf{v}_{k}=(a^{(1)}_{k},\dots,a^{(d)}_{k})\in\R^{d}$ for $k\in\N_{0}$.  The vectors $\mathbf{v}_{0},\dots,\mathbf{v}_{d-1}$ are linearly independent and thus $\phi(\cdot)=\det(\mathbf{v}_{0},\dots,\mathbf{v}_{d-2},\,\cdot\, )$ is a non-zero functional. Put $\mathrm{n}^{\uparrow}(p):=\min\set{k\in\{d,d+1,\dots\}}{\phi(\mathbf{v}_{k-1})\phi(\mathbf{v}_{k})\leq 0}\in\N\cup\{\infty\}$ and $\mathrm{n}^{\downarrow}(p):=\mathrm{n}^{\uparrow}(\tilde{p})$, where $\tilde{p}(x)=x^{d}p(\frac{1}{x})$. Since there is $0\neq f\in(p\R[x])^{+}$ of degree $m$ if and only if $\mathbf{v}_{m}\in\gen{\mathbf{v}_{0},\dots,\mathbf{v}_{m-1}}_{\R^{+}}$ and since $\phi$ corresponds to a particular halfspace, it follows that $\mathrm{n}^{\uparrow}(p)\leq \deg_{+}(p)$. Further, as one easily checks, the number $\mathrm{n}^{\uparrow}(p)$ will be the same if we choose any other basis than $(a^{(1)}_{i})_{i=0}^{\infty},\dots,(a^{(d)}_{i})_{i=0}^{\infty}$ (e.g. such one where the initial values correspond to the identity matrix). We summarize now our conclusions into \ref{xp1} and \ref{xp2}.




\begin{proposition}\label{xp1}
Let $p\in\mathcal{D}^{+}(\R)$ be of degree $\geq 2$. Then $\max\{\mathrm{n}^{\uparrow}(p),\mathrm{n}^{\downarrow}(p)\}\leq \deg_{+}(p)$. Moreover, if $p\in\R^{+}[x]$ then
 $\mathrm{n}^{\uparrow}(p)=\mathrm{n}^{\downarrow}(p)=\deg_{+}(p)$.
\end{proposition}

\begin{proposition}\label{xp2}
Let $p=a_{d}x^{d}+a_{d-1}x^{d-1}+\cdots+a_{0}\in\mathcal{D}^{+}(\R)$ be of degree $d\geq 2$ and $(r_{n})_{n\in\N_{0}}\sub\R$ be a sequence given by a recurrence relation $a_{d}r_{n+d}+a_{d-1}r_{n+d-1}+\cdots+a_{0}r_{n}=0$ for $n\in\N_{0}$ and $r_{0}=\cdots=r_{d-2}=0$, $r_{d-1}=1$.

Then $\mathrm{n}^{\uparrow}(p)=\min\set{n\in\{d,d+1,\dots\}}{r_{n}\leq 0}$
\end{proposition}

We show now that $\mathrm{n}^{\uparrow}(p)$ and $\mathrm{n}^{\downarrow}(p)$ determine the value of $\deg_{+}(p)$ for $p\in\mathcal{D}^{+}(\R)$ of degree $2$ or $3$.

\begin{lemma}\label{xp3}
Let $0<\varphi<\pi/2$, $0<s\leq 1$ and $n=\Big\lceil\frac{\pi}{\varphi}\Big\rceil$. Put $\mathbf{v}_{k}=(\cos k\phi,\sin k\phi, (-s)^{k})\in\R^{3}$ for $k\in\N_{0}$. Then $$(\mathbf{v}_{k}\times \mathbf{v}_{0})\cdot \mathbf{v}_{l}=\sin(k-l)\phi+(-s)^{k}\sin l\phi - (-s)^{l}\sin k\phi.$$

Let $H(\mathbf{u},\mathbf{v})=\set{\mathbf{w}\in\R^{3}}{(\mathbf{u}\times \mathbf{v})\cdot \mathbf{w}\geq 0}$ denote the halfspace determined by the (linearly independent) vectors $\mathbf{u},\mathbf{v}\in\R^{3}$. Similarly, put $H^{\circ}(\mathbf{u},\mathbf{v})=\set{\mathbf{w}\in\R^{3}}{(\mathbf{u}\times \mathbf{v})\cdot \mathbf{w}> 0}$. Let $\mathbf{e}_{3}=(0,0,1)\in\R^{3}$.
\begin{enumerate}
\item[(i)] $\mathbf{v}_{2},\dots,\mathbf{v}_{n-1}\in H^{\circ}(\mathbf{v}_{0},\mathbf{v}_{1})$.
\item[(ii)] If  $n$ is odd, then $H(\mathbf{v}_{0},\mathbf{v}_{1})\cap H(\mathbf{v}_{n-1},\mathbf{v}_{0})\cap H(\mathbf{v}_{1},\mathbf{e}_{3})\sub \gen{\mathbf{v}_{0},\dots,\mathbf{v}_{n-1}}_{\R^{+}}$.
\item[(iii)] If  $n$ is even, then $H(\mathbf{v}_{0},\mathbf{v}_{1})\cap H(\mathbf{v}_{n-2},\mathbf{v}_{0})\cap H(\mathbf{v}_{1},\mathbf{e}_{3})\sub \gen{\mathbf{v}_{0},\dots,\mathbf{v}_{n-1}}_{\R^{+}}$.
\end{enumerate}

\end{lemma}
\begin{proof}
The first equation is easy to verify. Further, for $l=2,\dots,n-1$ we have $(\mathbf{v}_{1}\times \mathbf{v}_{0})\cdot \mathbf{v}_{l}=-\sin(l-1)\phi\ -s\sin l\phi - (-s)^{l}\sin\phi\leq -s\sin l\phi - (\sin(l-1)\phi - \sin\phi)< 0$, hence $\mathbf{v}_{0},\dots,\mathbf{v}_{n-1}\in H^{\circ}(\mathbf{v}_{0},\mathbf{v}_{1})$.

 Let $\mathbf{u}=(\cos\psi,\sin\psi, r)\in H(\mathbf{v}_{0},\mathbf{v}_{1})$ be such that $\pi/2<\psi<\pi$ and $r\in\R$. Choose $0\neq \mathbf{w}\in\gen{\mathbf{v}_{0},\mathbf{u}}_{\R^{+}} \cap \gen{\mathbf{v}_{1},\mathbf{e}_{3}}_{\R^{+}}$. Then $H(\mathbf{u},\mathbf{v}_{0})=H(\mathbf{w},\mathbf{v}_{0})$ and $H(\mathbf{v}_{1},\mathbf{e}_{3})=H(\mathbf{v}_{1},\mathbf{w})$ and it is not difficult to show that $H(\mathbf{v}_{0},\mathbf{v}_{1})\cap H(\mathbf{u},\mathbf{v}_{0})\cap H(\mathbf{v}_{1},\mathbf{e}_{3})=H(\mathbf{v}_{0},\mathbf{v}_{1})\cap H(\mathbf{w},\mathbf{v}_{0})\cap H(\mathbf{v}_{1},\mathbf{w})=\gen{\mathbf{v}_{0},\mathbf{v}_{1},\mathbf{w}}_{\R^{+}}\sub \gen{\mathbf{v}_{0},\mathbf{v}_{1},\mathbf{u}}_{\R^{+}}$. Now, choosing either $\mathbf{u}=\mathbf{v}_{n-1}$ or $\mathbf{u}=\mathbf{v}_{n-2}$ we get the first inclusion in appropriate cases (i) and (ii).
\end{proof}

\begin{theorem}\label{3d}
Let $p\in\mathcal{D}^{+}(\R)$ be of degree $2$ or $3$. Then $\deg_{+}(p)=\max\{\mathrm{n}^{\uparrow}(p),\mathrm{n}^{\downarrow}(p)\}$.

Moreover, let $p=(x+c)(x-\omega)(x-\overline{\omega})\in\R[x]$, where $\omega=r(\cos \varphi+\I\sin\varphi)$, $c,r>0$ and $\varphi\in(0,\pi/2)$. Put $n=\Big\lceil\frac{\pi}{\varphi}\Big\rceil$, $0\leq \eps=n\varphi-\pi<\varphi$ and $s=\min\{\frac{c}{r},\frac{r}{c}\}$. Then
\begin{enumerate}
\item[(i)]  $n\leq \deg_{+}(p)\leq n+1$,
\item[(ii)] $\deg_{+}(p)=n$ if and only if $\sin(\phi-\eps)\leq s\cdot\sin\eps + (-1)^{n+1}s^{n}\cdot\sin\phi$.
\end{enumerate}
\end{theorem}
\begin{proof}
Using \ref{B4} and \ref{xp1} we only need to prove the case of $p=(x+c)(x-\omega)(x-\overline{\omega})$, where $\omega=r(\cos \varphi+\I\sin\varphi)$, $c,r>0$ and $\varphi\in(0,\pi/2)$.
Now, (i) follows immediately from \ref{B4}. Further, with help of substitutions $x=t\cdot x'$, $t>0$ and $x=1/x'$ we have that $\deg_{+}(p)=\deg_{+}(q)$ and $\max\{\mathrm{n}^{\uparrow}(p),\mathrm{n}^{\downarrow}(p)\}=\max\{\mathrm{n}^{\uparrow}(q),\mathrm{n}^{\downarrow}(q)\}$ for $q=(x+s)(x-\frac{\omega}{r})(x-\frac{\overline{\omega}}{r})$, where $s=\min\{\frac{c}{r},\frac{r}{c}\}\leq 1$. Put $\mathbf{v}_{k}=\mathbf{v}_{k}(q)$ for $k\in\N_{0}$.

By \ref{xp3}(i) we have $\Big\lceil\frac{\pi}{\varphi}\Big\rceil=n\leq \mathrm{n}^{\uparrow}(q)$. If $\deg_{+}(q)=n$, then $\deg_{+}(q)=\mathrm{n}^{\uparrow}(q)$ by \ref{xp1} and $\det(\mathbf{v}_{0},\mathbf{v}_{1},\mathbf{v}_{n})\leq 0$ by \ref{xp3}(i).
On the other hand, if $0\geq\det(\mathbf{v}_{0},\mathbf{v}_{1},\mathbf{v}_{n})=(\mathbf{v}_{0}\times \mathbf{v}_{1})\cdot \mathbf{v}_{n}$, then we need to show that $-\mathbf{v}_{n}\in \gen{\mathbf{v}_{0},\dots,\mathbf{v}_{n-1}}_{\R^{+}}$. Let $H(\cdot,\cdot)$ and $\mathbf{e}_{3}$ have the same meaning as in \ref{xp3}. Clearly, $-\mathbf{v}_{n}\in H(\mathbf{v}_{0},\mathbf{v}_{1})\cap H(\mathbf{v}_{1},\mathbf{e}_{3})$.

 If $n$ is odd then $(\mathbf{v}_{n-1}\times \mathbf{v}_{0})\cdot (-\mathbf{v}_{n})= \sin\phi + s^{n-1}\sin\eps - s^{n}\sin(\phi-\eps)\geq 0$ by \ref{xp3}, hence $-\mathbf{v}_{n}\in H(\mathbf{v}_{n-1},\mathbf{v}_{0})$.

If $n$ is even then $(\mathbf{v}_{n-2}\times \mathbf{v}_{0})\cdot (-\mathbf{v}_{n})= \sin 2\phi+s^{n-2}\sin\eps+s^{n}\sin(2\phi-\eps)\geq 0$. Thus $-\mathbf{v}_{n}\in H(\mathbf{v}_{n-2},\mathbf{v}_{0})$.

Putting it all together we conclude with $-\mathbf{v}_{n}\in \gen{\mathbf{v}_{0},\dots,\mathbf{v}_{n-1}}_{\R^{+}}$ by \ref{xp3}(ii),(iii) and $\deg_{+}(q)=n$. The rest is now easy and the condition in (ii) is only an equivalent version of the inequality $(\mathbf{v}_{0}\times \mathbf{v}_{1})\cdot \mathbf{v}_{n}\leq 0$.
\end{proof}

\begin{remark}\label{ex4}
 The lower estimation $\mathrm{n}^{\downarrow}(p)$ provides no additional information (i.e. $\mathrm{n}^{\downarrow}(p)=\deg(p)$) in the case when both the linear and constant terms have coefficients of the same (non-zero) signs. In particular, the equality in \ref{3d} does not hold for degree 4. For instance, for $p=(x+2)(x+1/2)(x-\omega)(x-\overline{\omega})$, $\omega=\cos\phi+\I\sin\phi$, is $\mathrm{n}^{\downarrow}(p)=\mathrm{n}^{\uparrow}(p)=4$, by \ref{xp2}, while $\deg_{+}(p)\geq \Big\lceil\frac{\pi}{\varphi}\Big\rceil$.

Apart of this, the result in \ref{3d} (together with \ref{xp2}) provides a simple way (by choosing particular halfspaces) how to determine the value of $\deg_{+}$ and we ask therefore for a possible generalization for polynomials of higher degrees.
\end{remark}

For an estimation of $\deg_{+}$ for a general polynomial is useful to investigate this value for powers of polynomials. Eaton \cite{eaton}  provided $\deg_{+}$ in the case when a quadratic polynomial  vanishes on a root of unity (see \ref{corol}). Following basically the same idea we can easily get the  estimation in \ref{powerpol}.

\begin{theorem}\cite{eaton}\label{corol}
Let $n,k\in\N$ and $p=(x-e^{\I\pi/n})(x-e^{-\I\pi/n})\in\R[x]$. Then $\deg_{+}(p^{k})= k\deg_{+}(p)=kn$.
\end{theorem}

\begin{theorem}\label{powerpol}
 Let $p\in\R[x]\setminus\R^{+}[x]$ be a monic quadratic polynomial, i.e. $p=(x-\omega)(x-\overline{\omega})$, where $\omega=r(\cos \varphi+\I\sin\varphi)$, $r>0$ and $\varphi\in(0,\frac{\pi}{2})$. Let $2\leq k\in\N$.
Then
\begin{equation}\label{A3}
\Big\lceil\frac{k\pi}{\varphi}\Big\rceil\leq \deg_{+}(p^{k})\leq k\Big\lceil\frac{\pi}{\varphi}\Big\rceil
\end{equation}
 In particular, $k\deg_{+}(p) -k+1\leq \deg_{+}(p^{k})\leq k\deg_{+}(p).$
\end{theorem}

The next proposition shows that both the upper and lower estimations in \ref{powerpol} can be reached.

\begin{proposition}\label{power-ex-proposition}
For every $n\geq 3$, $k\geq 2$ there are monic quadratic polynomials $p_{1},p_{2}$ such that

(i) $\deg_{+}(p_{1})=n$ and $\deg_{+}(p_{1}^{k})=k\deg_{+}(p_{1})$ and

(ii) $\deg_{+}(p_{2})=n$ and $\deg_{+}(p_{2}^{k})=k\deg_{+}(p_{2}) -k+1$.
\end{proposition}
\begin{proof}
(i) follows from \ref{corol}.
For (ii) we first take $\widetilde{p}_{2}=(x-e^{\I\pi/(n-1)})(x-e^{-\I\pi/(n-1)})$. By \ref{B4} there is $f=\sum^{n}_{i=0}a_{i}x^{i}$ divisible by $\widetilde{p}_{2}$ such that $a_{i}>0$ for every $i=0,\dots,n+1$. Now $g=(x^{n-1}+1)^{k-1}f$ is divisible by $\widetilde{p}_{2}^{k}$, $\deg(g)=kn-k+1$ and every $i$-th coefficient of $g$ is positive for $i=0,\dots,\deg(g)$. Using continuity we get $\varphi\in\R$ such that $\Big\lceil\frac{\pi}{\varphi}\Big\rceil=n$, $\Big\lceil\frac{k\pi}{\varphi}\Big\rceil=k(n-1)+1$ and $(g/\widetilde{p}_{2}^{k})p_{2}^{k}\in\R^{+}[x]$, where $p_{2}=(x-e^{\I\varphi})(x-e^{-\I\varphi})$. Hence, by \ref{powerpol}, $\deg_{+}(p_{2}^{k})=k\deg_{+}(p_{2}) -k+1$.
\end{proof}

\begin{example}\label{power-ex}
Even if the lower and upper bound in \ref{powerpol}(\ref{A3}) for a given polynomial differ, it may happen that the upper bound is achieved. Taking $p=x^2-x+3$, we have $5$ the lower- and $6$ the upper bound for $p^2$ in \ref{powerpol}(\ref{A3}). But $\deg_{+}(p^2)=6$, since $(x+c)p^{2}= x^5 + (c-2)x^4 + (7-2c)x^3 +(7c- 6)x^2 + (9-6c)x + 9c\notin\R^{+}[x]$ for any $c>0$.
\end{example}

The previous examples show that knowing only $\Big\lceil\frac{k\pi}{\varphi}\Big\rceil$ and $k\Big\lceil\frac{\pi}{\varphi}\Big\rceil$ is generally not enough for determining $\deg_{+}(p^{k})$ in \ref{powerpol}. Nevertheless, based upon numerical computing, we conjecture that a better estimation than \ref{powerpol}(\ref{A3}) might be true, namely
$\deg_{+}(p^{k})\leq \Big\lceil\frac{k\pi}{\varphi}\Big\rceil+1$
for $p$ a quadratic polynomial.

Now we turn to the general case. Using formal power series, Tur\'an \cite{turan} has shown, that for a polynomial $p\in\R[x]$ of degree $d$ with zeros outside the angle-domain $\set{0\neq z\in\C}{|\mathrm{arc} (z)|<\varphi_{0}\leq \pi/2}$ is
\begin{equation}\label{equa2}
\deg_{+}(p)\leq \frac{d}{2}\Big(\Big\lfloor\frac{\pi}{\varphi_{0}}\Big\rfloor+1\Big).
\end{equation}
In the next theorem we make a slight improvement of this result and provide a simpler proof.

\begin{theorem}\label{mahler}
Let $p\in\mathcal{D}^{+}(\R)\setminus\R^{+}[x]$ and $d=\deg(p)$. Let $\varphi_{0}$ be the least positive argument of all non-zero roots of $p$. Then $$\deg_{+}(p)\leq\Big\lfloor \frac{d}{2}\Big\rfloor\Big(\Big\lceil\frac{\pi}{\varphi_{0}}\Big\rceil-2\Big)+d.$$
\end{theorem}
\begin{proof}
Let $p=p_{1}\dots p_{r}q_{1}\dots q_{s}$ be a decomposition into irreducible polynomials in $\R[x]$, where $r\geq 1$, $s\geq 0$ and $p_{i}\in\R[x]\setminus\R^{+}[x]$, $q_{j}\in\R^{+}[x]$ for $i=1,\dots,r$, $j=1,\dots,s$. By \ref{B4} we have $\deg_{+}(p)\leq (d-2r) +\sum_{i=1}^{r}\deg_{+}(p_{i})\leq (d-2r) + r\lceil\frac{\pi}{\varphi_{0}}\rceil\leq \lfloor \frac{d}{2}\rfloor(\lceil\frac{\pi}{\varphi_{0}}\rceil-2)+d$, since $r\leq\lfloor \frac{d}{2}\rfloor$.
\end{proof}

\begin{remark}
We show that the estimation in \ref{mahler} can be reached for some polynomials, while the bound (\ref{equa2}) is bigger in these cases.

 Let $k,n\in\N$. Put $p_{1}=\big((x-e^{\I\pi/n})(x-e^{-\I\pi/n})\big)^{k}\in\R[x]$, $\varphi_{0}=\pi/n$ and $d_{1}=\deg(p_{1})$. Then, by \ref{corol}, is $\deg_{+}(p_{1})=kn=\Big\lfloor \frac{d_{1}}{2}\Big\rfloor\Big(\Big\lceil\frac{\pi}{\varphi_{0}}\Big\rceil-2\Big)+d_{1}$, but is less than $k(n+1)=\frac{d_{1}}{2}\Big(\Big\lfloor\frac{\pi}{\varphi_{0}}\Big\rfloor+1\Big)$.

 Similarly, for $p_{2}=p_{1}(x+c)$, $d_{2}=\deg(p_{2})$ and $c>0$ big enough we have $D(p_{1})=D(p_{2})$ and hence $\deg_{+}(p_{2})=kn+1=\Big\lfloor \frac{d_{2}}{2}\Big\rfloor\Big(\Big\lceil\frac{\pi}{\varphi_{0}}\Big\rceil-2\Big)+d_{2}$, while the estimation in (\ref{equa2}) is  $(k+\frac{1}{2})(n+1)=\frac{d_{2}}{2}\Big(\Big\lfloor\frac{\pi}{\varphi_{0}}\Big\rfloor+1\Big)$.
\end{remark}

If a polynomial is given only by its coefficients, it may be difficult to determine arguments of its zeros. Therefore it is natural to look for an estimation that will use other features of a polynomial. In a recent paper Za\" imi \cite{zaimi} gave a bound of $\deg_{+}(p)$ (based on estimation of Dubickas \cite{dubickas1,dubickas2}), for $p$ being irreducible over $\Q[x]$, using the discriminant and the Mahler measure:
\begin{theorem}{\cite[1.1]{zaimi}}\label{B3}
Let $p=\prod_{i=1}^{d}(x-\alpha_{i})$ be an irreducible polynomial in $\Q[x]$ with $0\neq\alpha_{i}\in\C\setminus\R^{+}$ for every $i=1,\dots,d$.
Then $$\deg_{+}(p)< \frac{2d\pi}{\arcsin\Big(\frac{1}{M^{d-1}}\sqrt{\frac{|\Delta|}{d^{d+3}}}\Big)}$$
where $\Delta=\prod_{1\leq i<j\leq d}(\alpha_{i}-\alpha_{j})^2$ is the discriminant of $p$ and $M=\prod_{i=1}^{d}\max\{1,|\alpha_{i}|\}$ is the Mahler measure of $p$.
\end{theorem}

Following the idea of \cite{zaimi}, we improve this result to get approximately 4 times better bound. The next lemma is a corollary of a theorem of Dubickas \cite{dubickas1}.

\begin{lemma}\cite{dubickas1}\label{dubickas}
Let $p=\prod\limits_{i=1}^{d}(x-\alpha_{i})\in\R[x]$ be a polynomial with $\alpha_{1},\dots,\alpha_{d}\in\C$ such that $|\alpha_{1}|\geq|\alpha_{2}|\geq\cdots\geq|\alpha_{d}|$. Let $\omega\in\C\setminus\R^{+}$ be a root of $p$. Then
$$\frac{|\omega-\overline{\omega}|}{2|\omega    |}\geq\frac{1}{T}\sqrt{\frac{|\Delta|}{d^{d+3}}}$$
where $\Delta$ is the discriminant of $p$ and $T=|\alpha_{1}|^{d-1}|\alpha_{2}|^{d-2}\dots|\alpha_{d-1}|$.
\end{lemma}

If $\varphi_{0}$ is the least positive argument of all the roots of $p$ from \ref{dubickas} and $\varphi_{0}\in(0,\pi/2)$, then $T\leq(\prod_{i=1}^{d}\max\{1,|\alpha_{i}|\})^{d-1}=M^{d-1}$ and we have $\sin \varphi_{0}\geq\frac{1}{M^{d-1}}\sqrt{\frac{|\Delta|}{d^{d+3}}}$ by \ref{dubickas}. The next corollary now follows by (\ref{mahler}).

\begin{corollary}\label{estim}
Let $p=\sum_{i=0}^{d}a_{i}x^{i}\in\mathcal{D}^{+}\setminus\R^{+}[x]$ be a monic squarefree polynomial, $d=\deg(p)$, $p(0)\neq 0$, $\Delta$ be the discriminant and $M$ be the Mahler measure of $p$. Moreover, let $L$ be an upper bound for the magnitude of any root of $p$.  Then
$$\deg_{+}(p)\leq\Big\lfloor \frac{d}{2}\Big\rfloor\bigg(\bigg\lceil\frac{\pi}{\arcsin\Big(\frac{1}{K^{d-1}}\sqrt{\frac{|\Delta|}{d^{d+3}}}\Big)}\bigg\rceil-2\bigg)+d$$
for $K\in\{M, L^{d/2}\}$.
\end{corollary}

As a bound in \ref{estim} may serve for instance $L=1+\max\{|a_{n-1}|,\dots,|a_{0}|\}$ or $L=\sqrt{1+\sum_{i=0}^{d-1}a_{i}^{2}}$.

Of course, the previous assertion depends heavily on the stage of knowledge about estimation of the distance of roots of a polynomial. Moreover, it may happen that the bound in \ref{dubickas} will be depressed by a small angle between other roots that do not influence the value of $\deg_{+}(p)$. However, so far \ref{estim} is the best known criterion.


\section{An unusual property}\label{sec3}

In the end let us introduce a conjecture that was inspired by an approach using formal power series. For a formal power serie $P=\sum_{i=0}^{\infty}a_{i}x^{i}\in\R[[x]]$ and $n\in\N_{0}$ denote $(P)_{n}=\sum^{n}_{i=0}a_{i}x^{i}\in\R[x]$.  As a consequence of a proof in \cite{turan} we get that for an irreducible $p\in\R[x]$ of degree 2 is $p\cdot (\frac{1}{p})_{n}\in\R^{+}[x]$ for $n=\deg_{+}(p)-2=D(p)$.

On the other hand, for $q=x^{2}-x+\frac{1}{2}$ is $\deg_{+}(q^{2})=8$, $D(q^{2})=4$ by \ref{powerpol} and $q^{2}\cdot(\frac{1}{q^{2}})_{4}\notin\R^{+}[x]$ since $\frac{1}{q^{2}}=4+16x+32x^2+32x^3-16x^4+\cdots$. However, we still have $q^{2}\cdot(\frac{1}{q})_{2}\cdot(\frac{1}{q})_{2}\in\R^{+}[x]$,  since $q\cdot(\frac{1}{q})_{2}\in\R^{+}[x]$. With support of similar computational experiments we formulate a conjecture.
We say that a polynomial $p\in\mathcal{D}^{+}(\R)$ such that $p(0)\neq 0$ \emph{has a property $(\ast)$} if there are $p_{1},\dots,p_{k}\in\R[x]$ such that $p=p_{1}\cdots p_{k}$, $p_{i}\cdot (\frac{1}{p_{i}})_{D(p_{i})}\in\R^{+}[x]$ for every $i=1,\dots,k$ and $D(p)=\sum_{i=1}^{k}D(p_{i})$.

\

\textbf{Conjecture:} Let $p\in\mathcal{D}^{+}(\R)$ be a quadratic polynomial and $p(0)\neq 0$. Then $p^{k}$ has the property $(\ast)$ for every $k\in\N$.

\begin{example}
This example shows a cubic polynomial $p=(x+\frac{1}{2})(x^2-\frac{3}{2}x+1)$ does not have the property $(\ast)$. Assume the contrary. By \ref{B4}, is $\deg_{+}(q)=5$ for $q=x^2-\frac{3}{2}x+1$ and thus we have $\deg_{+}(p)=5$, since $(x^{2}+x+\frac{3}{4})p=x^{5}+\frac{11}{16}x+\frac{3}{8}\in\R^{+}[x]$.  Hence $D(p)=2\neq 0+3=D(x+\frac{1}{2})+D(q)$. Therefore must be $p\cdot(\frac{1}{p})_{D(p)}\in\R^{+}[x]$ by assumption. But $\frac{1}{p}=2-x+\frac{9}{2}x^2-\frac{33}{4}x^3+\cdots$ and $p\cdot(\frac{1}{p})_{2}=\frac{9}{2}x^5-\frac{11}{2}x^4+\frac{33}{8}x^3+1\notin\R^{+}[x]$, a contradiction.

\end{example}



\end{document}